\documentclass[12pt,a4paper, reqno]{amsart}
\usepackage[abbrev]{amsrefs}
\usepackage{amssymb,mathrsfs}
\usepackage{mathtools}

\usepackage{etoolbox}

\makeatletter
\patchcmd{\@settitle}{\uppercasenonmath\@title}{}{}{}
\patchcmd{\@setauthors}{\MakeUppercase}{}{}{}
\patchcmd{\section}{\scshape}{}{}{}
\makeatother

\theoremstyle{plain}
 \newtheorem*{theorem}{Theorem}
\theoremstyle{definition}
 \newtheorem*{problem}{Problem}
 \newtheorem*{remark}{Remark}
 \newtheorem*{ack}{Acknowledgment}

\def\vol{\mathop{\mathrm{vol}}\nolimits}

\begin{document}
\title[]{\large  An example of an optimal transport map on a sphere}
\author[]{Asuka Takatsu$^{\dagger}$}
\thanks{$\dagger$
Graduate School of Mathematical Sciences, The University of Tokyo,
3-8-1 Komaba, Meguro-ku, Tokyo 153-8914, Japan;
 RIKEN Center for Advanced Intelligence Project (AIP), 
Nihonbashi 1-chome Mitsui Building, 15th floor,
1-4-1 Nihonbashi, Chuo-ku, Tokyo 103-0027, Japan\\
\qquad{\sf asuka-takatsu@g.ecc.u-tokyo.ac.jp}}
\maketitle
\vspace{-30pt}
\begin{abstract}
We construct an optimal transport map  between 
uniform distributions on 
two antipodal annuli  of a sphere.\\
\end{abstract}
Let $(M, \langle \cdot, \cdot \rangle_M) $ be a connected, complete, smooth Riemannian manifold
of dimension $n\in \mathbb{N}$.
We denote by $d_M$ and $\vol_M$ the Riemannian distance function and the Riemannian volume measure on~$M$, respectively.
For $p\in [1,\infty)$,
let 
$\mathcal{P}_p(M)$ denote the set of  Borel probability measures $\mu$ on $M$ satisfying 
\[
\int_{M}d_M(x_0,x)^p d\mu(x)<\infty
\quad
\text{for some $x_0\in  M$}.
\]
For $\mu,\nu \in \mathcal{P}_p(M)$, 
a Borel set $E\subset M$ with $\mu(E)=1$,   and  a Borel map $T\colon E\to M$,
the expression $T_\sharp \mu =\nu$
means that 
\[
\int_M h(y) d\nu(y)=
\int_M h(T(x)) d \mu(x) \quad \text{holds for any } h\in C_b(M).
\]
The Monge problem is formulated as follows.
\begin{problem}
For $\mu,\nu \in \mathcal{P}_p(M)$, 
solve 
\[
\inf\left\{\int_{M}d_M(x,T(x))^pd\mu(x)\biggm| T\colon M\to M \text{ is Borel such that }T_\sharp \mu=\nu\right\}
\]
and find a minimizer if it exists.
\end{problem}
If $\mu$ is absolutely continuous with respect to $\vol_M$,
then the Monge Problem has a minimizer~$T$,
which is uniquely determined $\mu$-almost everywhere 
(see \cite{Figalli}*{Theorem 1.1} for $p=1$ and \cite{FF}*{Theorem 1.1} for $p\in (1,\infty)$).
A minimizer is called a \emph{$p$-optimal transport map} from $\mu$ to~$\nu$.
There are very few cases for which a $p$-optimal transport map can be determined explicitly  even for $M=\mathbb{R}^n$:
On $\mathbb{R}$, a $p$-optimal transport map is expressed as  the cumulative distribution functions  (e.g. see~\cite{Sam}*{Chapter~2}).
On $\mathbb{R}^n$ with $n\in \mathbb{N}$, if $\nu$ is a translation (resp.\,a dilation)  of $\mu$, then  the translation (resp.\,the dilation) becomes a $p$-optimal transport map from $\mu$ to $\nu$.
Moreover, for $p=2$, 
we can find a $2$-optimal transport map between elliptical distributions determined by the same density function (see~\cite{DL}*{\S3}).

We give an example of $\mu,\nu\in \mathcal{P}_p(\mathbb{S}^n)$ for which 
a $p$-optimal transport map is determined explicitly.
Since this  is a simple example, related to \cite{Vi}*{Exercise~7.39(v)},
it may already be known.
However no reference could be found.

Let us  fix our notation.
We denote by 
$\langle \cdot,\cdot \rangle$ and~$|\cdot|$
the Euclidean inner product and the Euclidean norm, respectively.
Let  $(e_1,\cdots, e_{n+1})$ be the natural basis of $\mathbb{R}^{n+1}$.
For $z\in \mathbb{S}^n$,  set 
\begin{align*}
\theta_z\coloneqq \arccos \langle z, e_{n+1}\rangle=d_{\mathbb{S}^n}(z,e_{n+1}),
\qquad
z^{\perp}\coloneqq z-\langle  z, e_{n+1}\rangle e_{n+1}.
\end{align*}
For a set $E$, its characteristic function is denoted by  $\mathbf{1}_E$.
\begin{theorem}
Let $n\in \mathbb{N}$ and $p\in [1,\infty)$.
Fix 
$a,b,\alpha,\beta \in [0, \pi]$ with $a<\alpha$, $b<\beta$.
Set  
\begin{align*}
A&\coloneqq \{ x\in \mathbb{S}^{n} \mid d_{\mathbb{S}^{n}}(x, e_{n+1})\in (a,\alpha)  \},
&&
\mu\coloneqq \frac{\mathbf{1}_{A}}{\vol_{\mathbb{S}^{n}}(A)}\vol_{\mathbb{S}^{n}},\\
B&\coloneqq \{ y\in \mathbb{S}^{n}\mid d_{\mathbb{S}^{n}}(y, -e_{n+1})\in (b,\beta)  \},
&&
\nu\coloneqq \frac{\mathbf{1}_{B}}{\vol_{\mathbb{S}^{n}}(B)}\vol_{\mathbb{S}^{n}}.
\end{align*}
Let $\varphi \in C([a,\alpha])\cap C^\infty((a,\alpha))$
be  a strictly decreasing function satisfying 
\[
\frac{1}{\vol_{\mathbb{S}^{n}}(A)}
\int_a^\theta \sin^{n-1} rdr
=\frac{1}{\vol_{\mathbb{S}^{n}}(B)}
\int_{\varphi(\theta)}^\beta\sin^{n-1} rdr
\quad \text{for }\theta\in (a,\alpha).
\]
Then a map $T\colon A\to \mathbb{S}^{n}$ defined by
\[
T(x)=\frac{\sin \varphi(\theta_x)}{\sin \theta_x}x^\perp-\cos\varphi(\theta_x) e_{n+1}
\quad \text{for }x\in A
\]
is a  $p$-optimal transport map from $\mu$ to $\nu$.
\end{theorem}
\begin{proof}
Firstly, we show the bijectivity of $T\colon A\to B$.
Let $\tilde{\theta}\colon [b,\beta]\to [a,\alpha]$ be the inverse function of $\varphi\colon [a,\alpha]\to [b,\beta]$.
Define a map $S\colon B \to \mathbb{S}^{n}$ by 
\[
S(y)\coloneqq \frac{\sin \tilde{\theta}(\tilde{\varphi}_y) }{\sin \tilde{\varphi}_y} y^\perp+\cos \tilde{\theta}(\tilde{\varphi}_y)e_{n+1}
\quad \text{for }y\in B,
\]
where $\tilde{\varphi}_y\coloneqq \arccos \langle y, -e_{n+1}\rangle$.
Then we see that $T(A)=S, S(B)=A$, and 
\[
S(T(x))=x\quad  \text{for }x\in A,\qquad
T(s(y))=y\quad  \text{for }y\in B.
\]
Thus $T\colon A\to B$ is bijective.

Secondly, 
we  prove $T_\sharp \mu=\nu$.
Let $F\colon \mathbb{S}^n\setminus\{e_{n+1}\} \to \mathbb{R}^n$ 
and $G\colon  \mathbb{R}^n\to \mathbb{S}^n\setminus\{e_{n+1}\}$
be the stereographic projection from $e_{n+1}$ 
and its inverse map, respectively, that is, 
\[
F(z)\coloneqq \frac{z^{\perp}}{1-\langle z, e_{n+1}\rangle }
\quad \text{for }z\in \mathbb{S}^n\setminus\{e_{n+1}\},
\qquad
G(w)\coloneqq \frac{\left(2w, -1+|w|^2\right)}{1+|w|^2}
\quad \text{for }w\in \mathbb{R}^n,
\]
where we regard $z^\perp$ as an element in $\mathbb{R}^n$.
Then we have 
\begin{align}\label{want}
\frac{2^n}{\vol_{\mathbb{S}^n}(A)}
\int_{F(A)} \frac{h(T(G(u)))}{ (1+|u|^2)^{n}}du
=
\int_{\mathbb{S}^n} h\circ T d\mu
=
\int_{\mathbb{S}^n} hd\nu
=
\frac{2^n}{\vol_{\mathbb{S}^n}(B)}
\int_{F(B)} \frac{h(G(v))}{(1+|v|^2)^{n}} dv
\end{align}
for $h\in C_b(\mathbb{S}^n)$.
Since the map 
$\Psi \coloneqq F\circ  T\circ  G \colon F(A) \to F(B)$ is smooth and bijective,
we apply the change of variables $v=\Psi(u)$ to have
\begin{align}
\begin{split}
\label{change}
\int_{F(B)}  \frac{h(G(v))}{(1+|v|^2)^{n}} dv
&=
\int_{F(A)} \frac{h(T(G(u)))}{(1+|\Psi(u)|^2)^{n}} \left|\det D\Psi(u) \right| du\\
&=\int_{F(A)} \frac{h(T(G(u)))}{ (1+|u|^2)^{n}} \left(\frac{1+|u|^2}{1+|\Psi(u)|^2}\right)^{n} \left|\det D\Psi(u) \right|du.
\end{split}
\end{align}
For $u\in F(A)$, it turns out that 
\[
\theta_{G(u)}=\arccos\frac{-1+|u|^2}{1+|u|^2}=2\arctan \frac{1}{|u|},
\qquad
G(u)^{\perp}=\frac{2u}{1+|u|^2}
=2\sin^2 \frac{\theta_{G(u)}}{2} \cdot u.
\]
This yields 
\begin{align*}
\Psi(u)
&=
 \frac{\sin \varphi(\theta_{G(u)})}{\sin \theta_{G(u)}}\cdot
\frac{ G(u)^{\perp}}{1+\cos \varphi(\theta_{G(u)})}
=
\tan \frac{\theta_{G(u)}}{2}
\tan \frac{\varphi(\theta_{G(u)})}{2}
 \cdot u,\\
1+|u|^2&=1+\frac{1}{\tan^2\frac{\theta_{G(u)}}{2}}=\frac{1}{\sin^2\frac{\theta_{G(u)}}{2}},
\qquad
1+|\Psi(u)|^2=1+\tan^2\frac{\varphi(\theta_{G(u)})}{2}=\frac{1}{\cos^2\frac{\varphi(\theta_{G(u)})}{2}}.
\end{align*}
Setting 
\[
\tau(\theta)\coloneqq 
\tan \frac{\theta}{2}\tan \frac{\varphi(\theta)}{2} 
\quad \text{for }\theta\in (a,\alpha),
\]
we calculate 
\begin{align*}
\tau'(\theta)
&=\frac{\tau(\theta)}{\sin \theta}\left[ 1-
\frac{\vol_{\mathbb{S}^n}(B)}{\vol_{\mathbb{S}^n}(A)}
\left(\frac{\sin \theta}{\sin \varphi(\theta)}\right)^n \right],\\
D\Psi(u)
&=\tau\left(\theta_{G(u)}\right)
\left\{
 I_n-\left[1-\frac{\vol_{\mathbb{S}^n}(B)}{\vol_{\mathbb{S}^n}(A)}\left(\frac{\sin \theta_{G(u)}}{\sin \varphi(\theta_{G(u)})}\right)^n\right]\frac{u\otimes u}{|u|^2}\right\}.
\end{align*}
Since we have 
\begin{align*}
D\Psi(u)u
&=\tau\left(\theta_{G(u)}\right)
\frac{\vol_{\mathbb{S}^n}(B)}{\vol_{\mathbb{S}^n}(A)}\left(\frac{\sin \theta_{G(u)}}{\sin \varphi(\theta_{G(u)})}\right)^n u,\\
D\Psi(u)w
&=\tau(\theta_{G(u)})w
\quad \text{for $w\in \mathbb{R}^n$ with $\langle w, u\rangle =0$},
\end{align*}
we conclude 
\begin{align*}
\det D\Psi(u)
&=\tau\left(\theta_{G(u)}\right)^n
\frac{\vol_{\mathbb{S}^n}(B)}{\vol_{\mathbb{S}^n}(A)}\left(\frac{\sin \theta_{G(u)}}{\sin \varphi(\theta_{G(u)})}\right)^n \\
&=
\frac{\vol_{\mathbb{S}^n}(B)}{\vol_{\mathbb{S}^n}(A)}\left(\frac{\sin^2 \frac{\theta_{G(u)}}{2}}{\cos^2 \frac{\varphi(\theta_{G(u)})}{2}}\right)^n
=\frac{\vol_{\mathbb{S}^n}(B)}{\vol_{\mathbb{S}^n}(A)}\left(\frac{1+|\Psi(u)|^2}{1+|u|^2}\right)^n.
\end{align*}
This with \eqref{change} implies \eqref{want}, consequently $T_\sharp \mu=\nu$ follows.

Lastly, 
we show the optimality of $T$.
By Theorem~\cite{Sam}*{Theorem~1.49},
it suffices to check
\[
\sum_{\ell=1}^Ld_{\mathbb{S}^n}(x_\ell, T(x_\ell))^p
\leq 
\sum_{\ell=1}^Ld_{\mathbb{S}^n}(x_{\ell+1}, T(x_\ell))^p
\quad
\text{holds for any $L\in \mathbb{N}$ and $\{x_\ell\}_{\ell=1}^L\subset A$,}
\]
where $x_{L+1}\coloneqq x_1$.
For $x,x'\in A$, 
we compute
\begin{align*}
\langle x', T(x)\rangle
&=\frac{\sin  \varphi(\theta_x)}{\sin \theta_x} \langle x'^\perp, x^\perp\rangle
-\cos \theta_{x'}\cos \varphi(\theta_{x})\\
&\leq \frac{\sin  \varphi(\theta_x)}{\sin \theta_x} |x'^\perp| |x^\perp|
-\cos \theta_{x'}\cos \varphi(\theta_{x})
=\cos(\pi- \theta_{x'}-\varphi(\theta_x)),
\end{align*}
where the equality holds if $x=x'$.
Thus we see that
\begin{align}
\begin{split}
\label{ineq}
\sum_{\ell=1}^Ld_{\mathbb{S}^n}(x_\ell, T(x_\ell))^p
&=
\sum_{\ell=1}^L
(\pi- \theta_{x_\ell}-\varphi(\theta_{x_\ell}))^p,\\
\sum_{\ell=1}^Ld_{\mathbb{S}^n}(x_{\ell+1}, T(x_\ell))^p
&\geq 
\sum_{\ell=1}^L
(\pi- \theta_{x_{\ell+1}}-\varphi(\theta_{x_\ell}))^p.
\end{split}
\end{align}
Define a strictly increasing function $\sigma\colon (a,\alpha)\to (\pi-\beta,\pi-b)$ by 
$\sigma(s)\coloneqq \pi- \varphi(s)$ for $s\in (a,\alpha)$.
Since we have
\[
\frac{\mathbf{1}_{(a,\alpha)}}{\alpha-a} \vol_{\mathbb{R}},
\sigma_{\sharp} \frac{\mathbf{1}_{(a,\alpha)}}{\alpha-a} \vol_{\mathbb{R}} \in \mathcal{P}_p(\mathbb{R})
\]
and the monotonicity 
\[
(\sigma(s)-\sigma(s'))(s-s') \geq 0
\quad \text{for }s,s'\in (a,\alpha),
\]
we observe  from \cite{Sam}*{Lemma~2.8, Theorem~2.9, Theorem~1.38} that 
\[
\sum_{\ell=1}^L  (s_\ell-\sigma(s_\ell) )^p
\leq 
\sum_{\ell=1}^L  (s_{\ell+1}-\sigma(s_\ell) )^p
\quad
\text{for any $\{s_\ell\}_{\ell=1}^L\subset (a,\alpha)$,}
\]
where $s_{L+1}\coloneqq s_1$.
This inequality with the choice $s_\ell=\theta_{x_\ell}$ with \eqref{ineq} 
provides 
\begin{align*}
\sum_{\ell=1}^Ld_{\mathbb{S}^n}(x_\ell, T(x_\ell))^p
&\leq 
\sum_{\ell=1}^Ld_{\mathbb{S}^n}(x_{\ell+1}, T(x_\ell))^p
\end{align*}
as desired.
Thus the proof is achieved.
\end{proof}
\begin{remark}
Let $A,B, \mu,\nu$, and $T$ be as in Theorem.
\setlength{\leftmargini}{20pt} 
\begin{enumerate}
\item
Let $\psi \colon [0,\infty)\to [0,\infty)$ be an increasing, convex function.
Then $s\mapsto \psi(|s|)$ is continuous and convex on $\mathbb{R}$.
Again by
\cite{Sam}*{Lemma~2.8, Theorem~2.9, Theorem~1.38},
$T$ is a minimizer~of 
\[
\inf\left\{\int_{\mathbb{S}^n} \psi\left( d_M(x,\widetilde{T}(x)) \right) d\mu(x)\biggm| \widetilde{T}\colon M\to M \text{ is Borel such that }\widetilde{T}_\sharp \mu=\nu\right\},
\]
where a minimizer is uniquely determined  $\mu$-almost everywhere
if $\psi$ is strictly convex on~$[0,\infty)$.
\item
If $a=\pi-\beta$ and $\alpha=\pi-b$,
then 
$A=B, \mu=\nu$, $\varphi(\theta)=\pi-\theta$, and $T=\mathrm{id}_A$.
\item
For $x\in A$ and $t \in [0,1]$, set 
\[
\theta_x(t) \coloneqq (1-t)\theta_x+t (\pi-\varphi(\theta_x)), \qquad
\gamma_x(t)\coloneqq  \frac{\sin \theta_x(t)}{\sin \theta_x}x^\perp+ \cos \theta_x(t) e_{n+1}.
\]
Then $\gamma_x\colon [0,1]\to M$ is a unique minimal geodesic from $x$ to $T(x)$.
For $t \in [0,1]$, 
defined a map  $T^t \colon A \to M$ by $T^t(x)\coloneqq \gamma_x(t)$
for $x\in A$, and set
\begin{align*}
\mu_t&\coloneqq T^t_\sharp \mu,\\
A^t&\coloneqq T^t(A)=\{z\in \mathbb{S}^n\mid  d_{\mathbb{S}^n}(z,e_{n+1})\in ((1-t)a+t (\pi-\beta), (1-t) \alpha+t (\pi-b))  \}.
\end{align*}
Then $T^0=\mathrm{id}_A$ and $T^1=T$ hold. 
Moreover, for  $t \in (0,1)$,  $\mu^t(A^t )=1$ but $\mu^t$ is not a uniform distribution on $A^t$ unless 
either $n=1$ or  $a=\pi-\beta, \alpha=\pi-b$.
The $1$-parameter family $(\mu^t)_{t\in [0,1]}$ is called a \emph{displacement interpolation} from $\mu$ to $\nu$.
\end{enumerate}
\end{remark}
\begin{ack}
The author would like to thank  Noboru Isobe and Shin-ichi Ohta for useful comments.
The author was supported in part by JSPS KAKENHI Grant Number 
24H00183,
24K21513,
26H01996.
\end{ack}

\end{document}